\documentclass[12pt,a4paper]{article}

\RequirePackage{etex}
\usepackage[utf8]{inputenc}
\usepackage{amsmath, amssymb, amsthm, amsfonts, mathrsfs, systeme}
\usepackage{bbm, commath}
\usepackage{multicol}
\usepackage{graphicx}
\usepackage{tikz-network}
\usetikzlibrary {positioning, patterns,calc}
\usetikzlibrary{graphs, graphs.standard,quotes}
\usepackage{authblk}
\usepackage{enumitem}
\usepackage[colorlinks=true,linkcolor=black,citecolor=black]{hyperref}

\usepackage{lineno}
\modulolinenumbers[1]

\newtheorem{thm}{Theorem}

\newtheorem{lem}{Lemma}
\newtheorem{cor}{Corollary}
\newtheorem{exm}{Example}
\newtheorem{rem}{Remark}

\def \e {{\mathbf e}}

\def \x {{\bf x}}
\def \y {{\bf y}}
\def \z {{\bf z}}
\def \w {{\bf w}}

\def \o {{\bf 0}}

\newcommand{\ob}[1]{\left(#1\right)}
\newcommand{\cb}[1]{\left\lbrace #1\right\rbrace}
\newcommand{\tb}[1]{\left[#1\right]}

\title{\it $q$-Laplacian State Transfer on Graphs with Involutions}

\author[1]{Swornalata Ojha}
\author[2]{Hiranmoy Pal}
\affil[1,2]{Department of Mathematics, National Institute of Technology Rourkela, India-769008.}

\date{\today}

\begin{document}
\maketitle


\begin{abstract}

We study the existence of state transfer with respect to the $q$-Laplacian matrix of a graph equipped with a non-trivial involution. We show that the occurrence of perfect state transfer between certain pair (or plus) states in such a graph is equivalent to the existence of vertex state transfer in a subgraph induced by the involution with potentials. This yields infinite families of trees with potentials and unicyclic graphs of maximum degree three that exhibit perfect pair state transfer. In particular, we investigate vertex and pair state transfer in edge-perturbed complete bipartite graphs, cycles, and paths with potentials only at the end vertices.\\


\end{abstract}

\newpage

\section{Introduction}
Continuous-time quantum walks \cite{farhi} constitute a fundamental framework for studying various quantum transport phenomena in quantum spin networks. Such networks are modeled by a graph $G$, where vertices correspond to qubits and edges represent the interactions between them. Let $G = (V, E, w)$ be a weighted graph with $n$ vertices, where the weight function $w: E \rightarrow \mathbb{R}^+$ assigns a positive real weight to each edge of $G$. The \emph{adjacency matrix} $A \in \mathbb{R}^{n \times n}$ associated to $G$ is defined by
\[
A_{ij} = 
\begin{cases} 
w(i,j), & \text{whenever } (i,j) \in E, \\
0, & \text{otherwise}.
\end{cases}
\]
The \emph{degree matrix} $\Delta$ is a diagonal matrix with $\Delta_{ii}$ equal to the sum of the entries in the $i$th row of the adjacency matrix. The \emph{Laplacian matrix} $L$ and \emph{signless Laplacian matrix} $Q$ associated to $G$ are defined by $L = \Delta - A$ and $Q = \Delta + A$, respectively. The \emph{$q$-Laplacian matrix} $\mathscr{L}$ of $G$, introduced in \cite{bapat} with a non-zero real parameter $q$, is defined by
\begin{equation}\label{E1}
 \mathscr{L} = q L - (q - 1) I + q (q - 1) (\Delta - I).
\end{equation}
The $q$-Laplacian matrix reduces to the Laplacian matrix when $q = 1$, and to the signless Laplacian matrix when $q = -1$. A continuous-time quantum walk on $G$ with respect to the $q$-Laplacian matrix $\mathscr{L}$ is described by the transition matrix
\[
U_{\mathscr{L}}(t) := \exp{\ob{i t \mathscr{L}}}, \quad \text{where } t \in \mathbb{R}.
\]
A quantum state is represented by a positive semidefinite matrix of trace one, commonly referred to as a \emph{density matrix}. A density matrix is called a \emph{real pure state} if all its entries are real and it has rank one. A real pure state can be expressed in the form $D_\x = \x \x^T$, where $\x \in \mathbb{R}^n$ is a unit vector. Perfect state transfer in quantum communication networks was first introduced by Bose \cite{bose}. We say that \emph{perfect state transfer (PST)} relative to the $q$-Laplacian matrix $\mathscr{L}$ occurs at $\tau \in \mathbb{R}^+$ between two density matrices $D_1$ and $D_2$ if
\[
D_2 = U_\mathscr{L}(\tau) D_1 U_\mathscr{L}(-\tau).
\]
In particular, PST between the real pure states $D_{\x}$ and $D_{\y}$ is equivalent to the existence of a complex scalar $\gamma \in \mathbb{C}$ such that
\[U_\mathscr{L}(\tau) \x = \gamma \y.
\]
In this case, we simply say that PST occurs between the vector states $\x$ and $\y$. A vector state $\x$ is said to be \emph{periodic} with respect to $\mathscr{L}$ if PST occurs from $\x$ to itself. For a vertex $a$ in $G$, the characteristic vector $\e_a$ represents a \emph{vertex state}. More generally, a state of the form
\[
\frac{1}{\sqrt{1 + s^2}} \left( \e_a + s \e_b \right)
\]
is called an \emph{$s$-pair state} for some nonzero real number $s$. In particular, when $s = -1$, the state is referred to as a \emph{pair state}, and when $s = 1$, it is called a \emph{plus state}. If $\x$ and $\y$ are both vertex states, pair states, plus states, or $s$-pair states, then we use the terms \emph{vertex PST}, \emph{pair PST}, \emph{plus PST}, or \emph{$s$-pair PST}, respectively, to describe perfect state transfer between $\x$ and $\y$. Since the occurrence of vertex PST is a rare phenomenon \cite{god2}, a relaxation called \emph{pretty good state transfer} (PGST) was introduced in \cite{god1, vin}. Subsequently, PGST between real pure states was considered in \cite{pal9}. A graph $G$ is said to exhibit PGST between two linearly independent states $\x$ and $\y$ if there exists a sequence of real numbers $\{\tau_k\}$ such that
\[
\lim_{k \to \infty} U_\mathscr{L}\ob{\tau_k} \x = \gamma \y,
\]
where $\gamma \in \mathbb{C}$ has unit modulus. The existence of PST between real pure states is \emph{monogamous}, as shown in \cite[Lemma 5.1(3)]{god25}: if a state $\x$ admits PST to both states $\y$ and $\z$, then necessarily $\y = \z$. In contrast, PGST does not adhere to this restriction, as demonstrated in \cite[Example 4.1]{pal5}. Numerous classes of graphs have been shown to admit vertex state transfer, including paths \cite{banchi,kay2,kirk2}, circulant graphs \cite{ang, bavsic, pal6, pal61, pal3}, Cayley graphs \cite{ber, che, pal2}, distance-regular graphs \cite{cou3}, edge-perturbed graphs \cite{bose1, god4, pal8}, blow-up graphs \cite{mon2}, and joins \cite{alvir, ang, kirk1}, among others. However, it is observed in~\cite{cou0} that vertex PST with respect to the Laplacian matrix does not occur in a tree, except in the case of a path on two vertices. Moreover, among all trees, only the paths on two or three vertices exhibit PST between their end vertices under continuous-time quantum walk with respect to the adjacency matrix \cite{cou2024}.

PST between pair states with respect to the Laplacian matrix was first studied in \cite{chen}, where the authors provided characterizations of pair PST for various classes of graphs, including paths, cycles, and certain graph constructions. Further studies on pair PST in Cayley graphs appear in \cite{cao2021, cao2022}. In \cite{pal9}, the author investigated state transfer properties of edge-perturbed graphs with clusters and constructed an infinite family of non-regular graphs with maximum valency five that exhibit PST with respect to the adjacency, Laplacian, and signless Laplacian matrices, occurring between the same pair states and at the same time. Kim et al. \cite{kim} further generalized the notion of pair state transfer to \emph{$s$-pair state transfer}, where $s$ is a nonzero complex number. They provided characterizations of perfect $s$-pair state transfer in complete graphs, cycles, and antipodal distance-regular graphs that admit vertex PST. In \cite{pal7}, a framework for constructing graphs that exhibit pair state transfer with respect to the adjacency matrix was presented, using isomorphic branches possessing vertex state transfer. Additionally, a complete characterization of cycles admitting pair PGST was provided. State transfer between real pure states was considered in \cite{god25}, providing three fundamental results: (i) every periodic real pure state $\x$ admits PST with another real pure state $\y$; (ii) every connected graph admits PST between real pure states; and (iii) for any pair of real pure states $\x$ and $\y$, and for any time $\tau$, there exists a real symmetric matrix $M$ such that $\x$ and $\y$ exhibit PST relative to $M$ at time $\tau$. It also established a classification of PST between real pure states in graphs such as cycles, paths, and complete bipartite graphs.

The concept of sedentary families of graphs was first introduced by Godsil \cite{God5}, and was later extended in \cite{mont} to include \emph{sedentary vertices}. Subsequently, the existence of \emph{$C$-sedentary real pure states} was investigated in \cite{pal9}. A state $\x$ in a graph $G$ is said to be \emph{$C$-sedentary} if there exists a constant $0 < C \leq 1$ such that
\begin{equation}\label{E2}
\inf_{t > 0} \left| \x^T U_\mathscr{L}(t) \x \right| \geq C.
\end{equation}
If equality holds in \eqref{E2}, then $\x$ is called \emph{sharply $C$-sedentary}. Two distinct vertices $u$ and $v$ in a graph $G$ are said to be \emph{twins} if the following three conditions hold: (i) their neighborhoods, excluding themselves coincide, that is, $N_G(u) \setminus \{u,v\} = N_G(v) \setminus \{u,v\}$; (ii) the edges $(u,z)$ and $(v,z)$ have equal weights for every $z \in N_G(u) \setminus \{u,v\}$; and (iii) the potential assigned to vertices $u$ and $v$ is the same. If $u$ and $v$ are twin vertices in $G$, then by \cite[Lemma 2.9]{mont2}, the vector $\e_u - \e_v$ is an eigenvector of $\mathscr{L}$. Consequently, the pair state $\frac{1}{\sqrt{2}}(\e_u - \e_v)$ is sharply $1$-sedentary. It is worth mentioning that a $C$-sedentary state does not exhibit PGST.

This paper is organized as follows. In section \ref{sec2}, we examine state transfer in graphs with a non-trivial involution. We establish a connection between pair (plus) and vertex state transfer in graphs with involutions. In section \ref{sec3}, we begin with an investigation of pair PST in edge-perturbed complete bipartite graphs, followed by an analysis of both vertex PST and pair PST in cycles with a single edge perturbation. We then demonstrate the occurrence of pair PST in cycles modified by the addition of extra edges and potentials. Finally, we investigate vertex and pair PST in paths with potentials placed only at the end vertices.

We present a few preliminary definitions and results before proceeding to the next section. Let $G$ be an undirected weighted graph on $n$ vertices, possibly containing loops. Denote by $\theta_1 < \theta_2 < \cdots < \theta_d$ the distinct eigenvalues of the $q$-Laplacian matrix $\mathscr{L}$ associated with $G$, and let $F_j$ denote the corresponding orthogonal projection. The spectral decomposition of the transition matrix $U_\mathscr{L}(t)$ is given by
\[U_\mathscr{L}(t) = \exp{\left( i t \mathscr{L} \right)} = \sum_{j=1}^d e^{i t \theta_j} F_j.\]
The matrix $U_\mathscr{L}(t)$ is both symmetric and unitary, and it can be expressed as a polynomial in $\mathscr{L}$. The \emph{eigenvalue support} of a state $\x$ relative to $\mathscr{L}$, denoted by $\Lambda_\x(\mathscr{L})$, is defined by
$\Lambda_\x(\mathscr{L}) = \left\{ \theta_j : F_j \x \neq 0 \right\}.$ A state $\x$ is called a \emph{fixed state} relative to $\mathscr{L}$ whenever $|\Lambda_\x(\mathscr{L})| = 1$. It can be observed that a fixed state does not admit PST. Two linearly independent states $\x$ and $\y$ are said to be \emph{strongly cospectral} if for each $\theta_j \in \Lambda_\x(\mathscr{L})$, either $F_j \x = F_j \y$ or $F_j \x = - F_j \y$. As shown in \cite[Lemma 5.1(1)]{god25}, strong cospectrality is a necessary condition for PST to occur between two real pure states $\x$ and $\y.$

A symmetric matrix $B$ is said to be \emph{positive semidefinite} if $x^T B x \geq 0$ for all nonzero vectors $x \in \mathbb{R}^n$. One can observe that both Laplacian and signless Laplacian matrices associated with a graph are positive semidefinite. It is shown in \cite[Corollary 3.5]{bapat} that the $q$-Laplacian matrix of a tree is positive semidefinite whenever $q \in [-1,1]$. We now recall a special case of \emph{Cauchy's interlacing theorem} where a symmetric matrix is perturbed with a positive semidefinite matrix of rank one.

\begin{thm}\cite{so}\label{T1}
Let $B$ and $C$ be two symmetric matrices. Let $\alpha_1 \geq \cdots \geq \alpha_n$ and $\gamma_1  \geq \cdots \geq \gamma_n $ be the eigenvalues of $B$ and $B+C,$ respectively. If $C$ is a positive semidefinite matrix of rank one, then
$$\gamma_1 \geq \alpha_1 \geq \gamma_2 \geq \cdots\geq \alpha_{n-1} \geq \gamma_n \geq \alpha_n.$$
\end{thm}

\section{State transfer on graphs with involutions}\label{sec2}
An involution of a graph is an automorphism of order two. Consider a graph $G$ equipped with a non-trivial involution $\phi: V(G) \to V(G)$ where $\phi^2(v)=v$ for every vertex $v.$ Suppose the vertices in $G$ are assigned a potential $\eta: V(G) \to \mathbb{R}$ that is symmetric under the involution $\phi$, that is, $\eta(v) = \eta(\phi(v))$ for all $v \in V(G).$ Let $S=\cb{v \in V(G) ~|~ \phi(v)=v},$ the set of all fixed vertices of $\phi.$ A half graph $G'$ induced by the involution $\phi$ is obtained by selecting exactly one vertex from each orbit $(v, \phi(v)),$ for all $v \in V(G)\backslash S.$ The graph $G'$ is an induced subgraph of $G$ on the chosen vertices with a potential inherited by restricting $\eta$ to the vertices of $G'$. One can observe that $G$ contains at least two copies of $G'$. The total number of vertices of $G$ equals $ 2n + |S|,$ where $|S|$ is the number of fixed vertices of $\phi,$ and $n$ is the number of vertices in $G'.$ Let $\Delta$ be the degree matrix of $G$ without potential. The degree matrix of $G$ with potential is considered as $\Delta+\Delta',$ where $\Delta '$ is a diagonal matrix with $\Delta_{vv}'=\eta(v).$ With a suitable labeling of the vertices of $G,$ the $q$-Laplacian matrix in \eqref{E1} can be written in block form as
\[\mathscr{L}= \ob{1-q^2}I+q^2(\Delta+\Delta')-qA=\begin{bmatrix}
	\mathscr{L}'& A_\phi& A_S\\
	A_\phi &\mathscr{L}' &A_S\\
	A_S^T & A_S^T & \mathscr{L}_S
\end{bmatrix},\]
where $\mathscr{L}'$ and $\mathscr{L}_S$ are the $q$-Laplacian matrices of the half graph $G'$ and the subgraph induced by $S,$ respectively, including the original degrees of vertices as in $G.$ The submatrix $A_\phi$ is a symmetric matrix corresponding to the edges across the involution, and the submatrix $A_S$ describes the edges between $S$ and the vertices of the half graph. For example, the $q$-Laplacian matrix of the wheel $W_5$ in Figure \ref{f1} is given by

\[
\mathscr{L} =
\left[
\begin{array}{c|c|c}
\begin{array}{cc} 2q^2+1 & -q \\ -q & 2q^2+1 \end{array} &
\begin{array}{cc} -q & 0 \\ 0 & -q \end{array} &
\begin{array}{c} -q \\ -q \end{array} \\
\hline
\begin{array}{cc} -q & 0 \\ 0 & -q \end{array} &
\begin{array}{cc} 2q^2+1 & -q \\ -q & 2q^2+1 \end{array} &
\begin{array}{c} -q \\ -q \end{array} \\
\hline
\begin{array}{cc} -q & -q \end{array} &
\begin{array}{cc} -q & -q \end{array} &
\begin{array}{c} 3q^2+1 \end{array}
\end{array}
\right].
\]
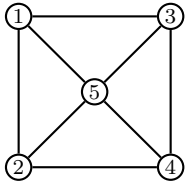
\begin{figure}
\centering
\begin{tikzpicture}[scale=1]
    \node[thick,circle,draw,inner sep=1pt] (C) at (0,0) {\scriptsize$2$};
    \node[thick,circle,draw,inner sep=1pt] (D) at (2,0) {\scriptsize$4$};
    \node[thick,circle,draw,inner sep=1pt] (E) at (0,2) {\scriptsize$1$};
    \node[thick,circle,draw,inner sep=1pt] (F) at (2,2) {\scriptsize$3$};
    \node[thick,circle,draw,inner sep=1pt] (A) at (1,1) {\scriptsize$5$};

    \draw[thick] (C) -- (D);
    \draw[thick] (C) -- (E);
    \draw[thick] (D) -- (F);
    \draw[thick] (E) -- (F);
    \draw[thick] (A) -- (D);
    \draw[thick] (A) -- (E);
    \draw[thick] (A) -- (F);
    \draw[thick] (A) -- (C);
\end{tikzpicture}
\caption{Wheel graph $W_5.$}
\label{f1}
\end{figure}
A spectral characterization relative to the adjacency matrix of a graph with involution is given in \cite[Lemma 2]{kempton}. We present an analogous observation for the $q$-Laplacian matrix of a graph with involution.
\begin{lem} \label{L1}
 Let $G$ be a graph with a non-trivial involution $\phi.$ Then the characteristic polynomial of the $q$-Laplacian matrix $\mathscr{L}$ of $G$ factors into $P_+(x)$ and $P_-(x)$ which are, respectively, the characteristic polynomials of $$\mathscr{L}_+:=\begin{bmatrix}
		\mathscr{L}'+A_\phi& A_S\\
		2A_S^T &\mathscr{L}_S
	\end{bmatrix}\quad \text{and}\quad \mathscr{L}_-:=\mathscr{L}'-A_\phi.$$
Furthermore, the eigenvectors of $\mathscr{L}$ take the block form $[a~~a~~b]^T$ and $[c~~ -c~~ \o]^T,$ where $[a~~ b]^T$ is an eigenvector of $\mathscr{L}_+$, and $c$ an eigenvector of $\mathscr{L}_-.$ 
 \end{lem}
 \begin{proof}
The matrix $\mathscr{L}_+$ is diagonalizable as  $\widetilde{\mathscr{L}_+}= K^{-1}\mathscr{L}_+ K,$ where $$\widetilde{\mathscr{L}_+}=\begin{bmatrix}
		\mathscr{L}'+A_\phi&\sqrt{2} A_S\\
		\sqrt{2}A_S^T & \mathscr{L}_S
\end{bmatrix}\quad \text{and} \quad K=\begin{bmatrix}
		I& \o\\
		\o & \sqrt{2}I
	\end{bmatrix}. $$ 
Suppose $\lambda$ is an eigenvalue of $\mathscr{L}_+$ with eigenvector $[a~~ b]^T$. Then $(\mathscr{L}'+A_\phi)a+A_Sb=\lambda a$ and $2A_S^Ta+\mathscr{L}_Sb=\lambda b.$ Hence
$$\begin{bmatrix}
		\mathscr{L}'& A_\phi& A_S\\
		A_\phi &\mathscr{L}' &A_S\\
		A_S^T & A_S^T & \mathscr{L}_S  
\end{bmatrix}\begin{bmatrix}
		a\\a\\b
	\end{bmatrix}=\begin{bmatrix}
		(\mathscr{L}'+A_\phi)a+A_Sb\\
		(\mathscr{L}'+A_\phi)a+A_Sb\\
		2A_S^Ta+\mathscr{L}_Sb
	\end{bmatrix}=\lambda \begin{bmatrix}
		a\\a\\b  
	\end{bmatrix}.$$  
It follows that $\lambda$ is also an eigenvalue of $\mathscr{L}$ with the same multiplicity as in $\mathscr{L}_+$. Consequently, if $P(x)$ is the characteristic polynomial of $\mathscr{L},$ then $P_+(x)$ divides $P(x).$ Similarly, suppose that $\mu$ is an eigenvalue of $\mathscr{L}_{-}$ with eigenvector $c,$ then $[c~~ -c~~ \o]^T$ is an eigenvector of $\mathscr{L}$ associated with the eigenvalue $\mu,$ with the same multiplicity as in $\mathscr{L}_{-}$. Accordingly, the polynomial $P_-(x)$ divides $P(x).$ Since both $P_+(x)$ and $P_-(x)$ are monic polynomials, and their degrees add up to the degree of the polynomial $P(x),$ it follows that $P(x)= P_+(x) P_-(x).$
\end{proof}
We now derive a relationship between the transition matrices $U_\mathscr{L}(t)$, $U_{\mathscr{L}_{-}}(t)$, and $U_{\widetilde{\mathscr{L}_+}}(t)$, thereby establishing a connection between PST in a subgraph to PST in the original graph in the presence of a non-trivial involution. The following result is analogous to the one stated in \cite[Lemma 4]{pal7}.
 \begin{lem}\label{L2}
 Let $G$ be a graph equipped with a non-trivial involution $\phi$, and $G'$ be a half graph induced by $\phi$. If $U_\mathscr{L}(t)$ is the transition matrix associated with the $q$-Laplacian matrix of $G,$ then for $u \in V(G')$,
 \begin{equation}\label{E3}
 U_\mathscr{L}(t)\ob{\e_u-\e_{\phi(u)}}=(I-P)\sum_{\mu_r \in \sigma(\mathscr{L}_-)}e^{it\mu_r}F_r\e_u,
 \end{equation}
 where P is the matrix of the involution $\phi$, and if $F_r'$ denotes the orthogonal projection of $\mathscr{L}_-$ corresponding to the eigenvalue $\mu_r,$ then the matrix $F_r$ is given by
 $$F_r=\frac{1}{2}\begin{bmatrix}
 F_r'&-F_r'&\o\\
 -F_r'& F_r'& \o\\
 \o& \o &\o
 \end{bmatrix}.$$
 \end{lem}
 \begin{proof}
The spectral decomposition of the transition matrix $U_\mathscr{L}(t)$ gives
  \begin{equation}\label{E4}
  U_\mathscr{L}(t)\ob{\e_u-\e_{\phi(u)}}=\sum_{\mu_r \in \sigma(\mathscr{L})}e^{it\mu_r}E_r\ob{\e_u-\e_{\phi(u)}},    
  \end{equation}
where $\sigma(\mathscr{L})$ denotes the $q$-Laplacian spectrum, and $E_r$ denotes the orthogonal projection associated with eigenvalue $\mu_r$. It follows from Lemma \ref{L1} that the only eigenvectors of $\mathscr{L}$ that contribute to the sum of \eqref{E4} are of the form $[c~~ -c~~ \o]^T,$ where $c $ is an eigenvector of $\mathscr{L}_-.$ Since $E_r$ is a polynomial in $\mathscr{L},$
\[E_r\ob{\e_u-\e_{\phi(u)}}=E_r(I-P)\e_u=(I-P)E_r\e_u=(I-P)F_r\e_u.\]
This completes the proof.
 \end{proof}
 We next establish that the transition matrix $U_\mathscr{L}(t)$ is similar to a block diagonal matrix where the transition matrices corresponding to  $\mathscr{L}_{-}$ and $\widetilde{\mathscr{L}_+}$ appear as the diagonal blocks.
\begin{thm}\label{T2}
Let $G$ be a graph equipped with a non-trivial involution $\phi$, and $G'$ be a half graph induced by $\phi$. Suppose $M$ is the matrix whose columns are the vectors in $\mathcal{B}= \cb{\frac{1}{\sqrt{2} }\ob{\e_u-\e_{\phi(u)} }~|~  u \in V(G')},$ $\mathcal{C}=\cb{\frac{1}{\sqrt{2} }\ob{\e_u+\e_{\phi(u)} }~|~ u \in V(G')},$ and $\mathcal{D}=\cb{\e_u~|~ u \in S},$ respectively. If $\mathscr{L}$ is the $q$-Laplacian matrix associated with $G$, then for every $t \in \mathbb{R},$
\[M^T U_\mathscr{L}(t) M= \begin{bmatrix}	U_{\mathscr{L}_-}(t) & \o\\	\o & U_{\widetilde{\mathscr{L}_+}}(t)
\end{bmatrix},\]
where $\mathscr{L}',~\mathscr{L}_{-},~\mathscr{L}_{S}, ~A_\phi, ~A_S$ are as defined previously, and
\[\widetilde{\mathscr{L}_+}=\begin{bmatrix}	\mathscr{L}'+A_\phi&\sqrt{2} A_S\\
\sqrt{2}A_S^T & \mathscr{L}_S
\end{bmatrix}.\]
\end{thm}
\begin{proof}
Let $\e_u$ and $\check{\e}_u$ be the characteristic vectors of a vertex $u$ in $G$ and $G',$ respectively. Since $U_\mathscr{L}(t)$ is a polynomial in $\mathscr{L},$ the permutation matrix $P$ associated with the involution $\phi$ commutes with $U_\mathscr{L}(t).$ Since $\ob{\e_v-\e_{\phi(v)}}=\ob{I-P}\e_v$, the following holds for every $u,v\in V\ob{G'}$
\begin{eqnarray*}
 \frac{1}{2}\ob{\e_v-\e_{\phi(v)}}^T U_\mathscr{L}(t)\ob{\e_u-\e_{\phi(u)}} 
 &=& \e_v^TU_\mathscr{L}(t)\ob{\e_u-\e_{\phi(u)}}.
\end{eqnarray*}
It follows from Lemma \ref{L2} that
\begin{eqnarray*}
  \e_v^TU_\mathscr{L}(t)\ob{\e_u-\e_{\phi(u)}} &=& \e_v^T(I-P)\sum_{\mu_r \in \sigma(\mathscr{L}_-)}e^{it\mu_r}~F_r\e_u\\
  &=& \frac{1}{2}\sum_{\mu_r \in \sigma(\mathscr{L}_-)}e^{it\mu_r}~ \begin{bmatrix}
      \check{\e}_v^T& -\check{\e}_v^T& \o
  \end{bmatrix} \begin{bmatrix}
   F_r'\\
   -F_r'\\
   \o
  \end{bmatrix}\check{\e}_u\\
  &=& \check{\e}_v^T U_{\mathscr{L}_-}(t) \check{\e}_u.
\end{eqnarray*}
Since $W = \text{span}(\mathcal{B})$ is an invariant subspace of $U_{\mathscr{L}}(t)$, the matrix representation of the restriction operator $U_{\mathscr{L}}(t)|_W$ relative to the basis $\mathcal{B}$ is the transition matrix $U_{\mathscr{L}_-}(t)$. Let $N$ be the matrix whose columns are the vectors in $\mathcal{C}$ and $\mathcal{D}$, respectively. Then
\[
N = \begin{bmatrix}
\frac{1}{\sqrt{2}} I_{|V(G')|} & \mathbf{0} \\
\frac{1}{\sqrt{2}} I_{|V(G')|} & \mathbf{0} \\
\mathbf{0} & I_{|S|}
\end{bmatrix}.
\]
It follows that $\mathscr{L}N=N\widetilde{\mathscr{L}}_+$, and therefore $U_\mathscr{L}(t)N=NU_{\widetilde{\mathscr{L}_+}}(t).$ Since the columns of $N$ span the orthogonal complement of $W$, we have the desired conclusion.
\end{proof}
The following observations are immediate from Theorem \ref{T2}.
\begin{cor}\label{C1}
Suppose the conditions of Theorem \ref{T2} hold. Then
\begin{enumerate}
\item Strong cospectrality (respectively, PST, PGST) relative to 
\[\begin{bmatrix}
    \mathscr{L}_- & \o\\ \o & \widetilde{\mathscr{L}}_+
\end{bmatrix}\]
occurs between two states $\x$ and $\y$ if and only if strong cospectrality (respectively, PST, PGST) relative to the $q$-Laplacian matrix of $G$ occurs between $M\x$ and $M\y.$
\item Periodicity (respectively, $C$-sedentariness) relative to \[\begin{bmatrix}
    \mathscr{L}_- & \o\\ \o & \widetilde{\mathscr{L}}_+
\end{bmatrix}\]
occurs from a state $\x$ if and only if periodicity (respectively, $C$-sedentariness) relative to the $q$-Laplacian matrix of $G$ occurs from $M\x$.
\end{enumerate}
\end{cor}

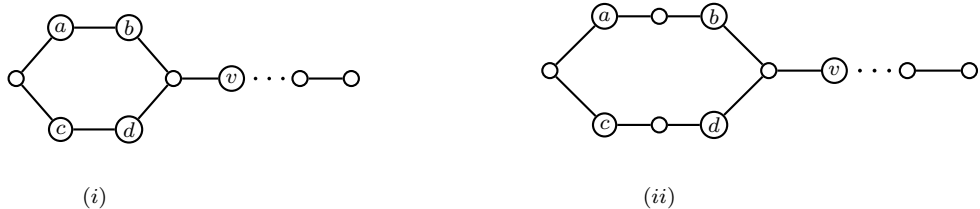
\begin{figure}
\centering

\begin{minipage}[t]{0.47\textwidth}
\centering

\begin{tikzpicture}[scale=0.45]
\node[circle,thick,draw,inner sep=2pt] (A) at (-1.3,0) {};
\node[circle,thick,draw,inner sep=1.3pt] (B) at (0,1.5) {\scriptsize$a$};
\node[circle,thick,draw,inner sep=1pt] (C) at (2,1.5) {\scriptsize$b$};
\node[circle,thick,draw,inner sep=1.3pt] (D) at (0,-1.5) {\scriptsize$c$};
\node[circle,thick,draw,inner sep=1pt] (E) at (2,-1.5) {\scriptsize$d$};
\node[circle,thick,draw,inner sep=2pt] (F) at (3.3,0) {};
\node[circle,thick,draw,inner sep=1.3pt] (G) at (5,0) {\scriptsize$v$};
\node[circle,thick,draw,inner sep=2pt] (I) at (7,0) {};
\node[circle,thick,draw,inner sep=2pt] (J) at (8.5,0) {};
\node at (1,-3.5) {\scriptsize$(i)$};

\draw[thick] (A) -- (B);
\draw[thick] (B) -- (C);
\draw[thick] (A) -- (D);
\draw[thick] (D) -- (E);
\draw[thick] (F) -- (E);
\draw[thick] (C) -- (F);
\draw[thick] (F) -- (G);
\draw[thick] (I) -- (J);
\foreach \x in {5.7,6.1,6.5} {\fill (\x,0) circle (1.5pt); }

\end{tikzpicture}
\end{minipage}
\begin{minipage}[t]{0.46\textwidth}
\centering

\begin{tikzpicture}[scale=0.48]
\node[circle,thick,draw,inner sep=2pt] (A) at (-3,0) {};
\node[circle,thick,draw,inner sep=2pt] (B) at (0,1.5) {};
\node[circle,thick,draw,inner sep=1pt] (C) at (1.5,1.5) {\scriptsize$b$};
\node[circle,thick,draw,inner sep=2pt] (D) at (0,-1.5) {};
\node[circle,thick,draw,inner sep=1pt] (E) at (1.5,-1.5) {\scriptsize$d$};
\node[circle,thick,draw,inner sep=2pt] (F) at (3,0) {};
\node[circle,thick,draw,inner sep=1.3pt] (G) at (4.8,0) {\scriptsize$v$};
\node[circle,thick,draw,inner sep=2pt] (I) at (6.8,0) {};
\node[circle,thick,draw,inner sep=2pt] (J) at (8.5,0) {};
\node[circle,thick,draw,inner sep=1.3pt] (K) at (-1.5,1.5) {\scriptsize$a$};
\node[circle,thick,draw,inner sep=1.3pt] (L) at (-1.5,-1.5){\scriptsize$c$};
\node at (0,-3.5) {\scriptsize$(ii)$};

\draw[thick] (A) -- (K);
\draw[thick] (B) -- (C);
\draw[thick] (A) -- (L);
\draw[thick] (K) -- (B);
\draw[thick] (D) -- (L);
\draw[thick] (D) -- (E);
\draw[thick] (F) -- (E);
\draw[thick] (C) -- (F);
\draw[thick] (F) -- (G);
\draw[thick] (I) -- (J);
\foreach \x in {5.5,5.9,6.3} {\fill (\x,0) circle (1.5pt); }

\end{tikzpicture}
\end{minipage}
\caption{Pair state transfer on cycles with tail.}
\label{f2}
\end{figure}
The graph $G$ in Figure \ref{f2}(i) is obtained by attaching a path of arbitrary length to a fixed vertex of the cycle $C_6$ on six vertices. The graph $G$ admits the involution $\phi=(a~~c)(b~~d)$. The $q$-Laplacian matrix of the half graph induced by $\phi$, including the original degrees of the vertices is given by
\[\mathscr{L}'=\ob{1+q^2}I-q A\ob{P_2},\]
where $A\ob{P_2}$ denotes the adjacency matrix of the path $P_2$. Since $A_{\phi}=\o$, it follows that $\mathscr{L}_{-}=\mathscr{L}'$. As PST relative to $\mathscr{L}_{-}$ occurs at time $\frac{\pi}{2q}$ between $\e_a$ and $\e_b$, it follows from Corollary \ref{C1} that $G$ exhibits PST relative to its $q$-Laplacian matrix between the states $\frac{1}{\sqrt{2}}\ob{\e_a-\e_c}$ and $\frac{1}{\sqrt{2}}\ob{\e_b-\e_d}$ at the same time. A similar observation holds for the graph obtained by attaching a path to a fixed vertex of the cycle $C_8$ as shown in Figure \ref{f2}$(ii)$, where pair PST relative to the $q$-Laplacian matrix occurs at time $\frac{\pi}{q\sqrt{2}}$ between the states $\frac{1}{\sqrt{2}}\left( \mathbf{e}_a - \mathbf{e}_c \right)$ and $\frac{1}{\sqrt{2}}\left( \mathbf{e}_b - \mathbf{e}_d \right)$. 
\begin{cor}
There are infinitely many unicyclic graphs of maximum degree three that exhibit $q$-Laplacian pair PST.
\end{cor}
In Figure \ref{f2}, if all vertices of $C_6$ (respectively, $C_8$) are connected to the vertex $v$, then it follows from Corollary \ref{C1} that the resulting graph exhibits pair PST with respect to its $q$-Laplacian matrix between $\frac{1}{\sqrt{2}}\left( \mathbf{e}_a - \mathbf{e}_c \right)$ and $\frac{1}{\sqrt{2}}\left( \mathbf{e}_b - \mathbf{e}_d \right)$. This fact is also observed in \cite[Theorem 1]{pal9} for both Laplacian and signless Laplacian matrices.

\section{Edge perturbed graphs}\label{sec3}
We investigate state transfer in special classes of graphs with respect to the $q$-Laplacian matrix, focusing in particular on the Laplacian and signless Laplacian cases. Although the complete graph $K_n$ does not exhibit pair PST---since every pair state remains fixed---removing an edge from $K_n$ induces pair PST relative to the Laplacian matrix, as shown in \cite[Corollary 5.4]{chen}. Moreover, it can be observed in \cite[Theorem 4]{pal9} that deleting a matching of size at least two from $K_n$ leads to pair PST with respect to the adjacency, Laplacian, and signless Laplacian matrices. The Corollary \ref{C1} further supports this result for the Laplacian and signless Laplacian cases. 

\subsection{Complete bipartite graphs}
 It is observed in \cite[Corollary 10.5]{god25} that a complete bipartite graph $K_{m,n}$ admits Laplacian pair PST if and only if either $m=n=2$ or $(m,n) \in \{(2,4k), (4k,2)\}$ for any integer $k \geq 1$.
 We investigate the existence of $q$-Laplacian pair PST in $K_{m,n}$ under certain edge perturbations. Let $M_k$ be a matching of size $k \geq 2$ in a complete bipartite graph $K_{m,n}$, with $\{a,c\}$ and $\{b,d\}$ among its edges, where $a,b$ belong to the same partite set. Consider the involution $\phi = (a~~b)(c~~d)$ on the graph $K_{m,n} - M_k$, obtained by removing all edges in $M_k$. When $m = n$, the matrix $\mathscr{L}_{-}$ is given by
\[
\mathscr{L}_{-} = \left( 1 + (m-2)q^2 \right) I - q A(P_2).
\]
Since PST with respect to $\mathscr{L}_{-}$ occurs at time $\frac{\pi}{2q}$, Corollary \ref{C1} implies that pair PST relative to the $q$-Laplacian matrix of $K_{m,m} - M_k$ occurs between $\frac{1}{\sqrt{2}} \left( \mathbf{e}_a - \mathbf{e}_b \right)$ and $\frac{1}{\sqrt{2}} \left( \mathbf{e}_c - \mathbf{e}_d \right)$ at the same time. When $m > n$, we add a set of edges $E$ within the larger partite set containing $a$ and $b$, connecting each of the vertices $a$ and $b$ to $ m-n$ additional vertices, chosen to preserve the involution $\phi$. By a similar argument, pair PST relative to the $q$-Laplacian matrix of the graph $K_{m,n} - M_k + E$ occurs at time $\frac{\pi}{2q}$ between the same pair states.

\begin{thm}\label{T3}
Let $M_k$ be a matching of size $k \geq 2$ in a complete bipartite graph $K_{m,n}$, with $\{a,c\}$ and $\{b,d\}$ among its edges, where $a$ and $b$ belong to the same partite set. Then, with respect to the $q$-Laplacian matrix, there is pair PST between the states $\frac{1}{\sqrt{2}}(\mathbf{e}_a - \mathbf{e}_b)$ and $\frac{1}{\sqrt{2}}(\mathbf{e}_c - \mathbf{e}_d)$ at time $\frac{\pi}{2q}$ in the following cases:
\begin{enumerate}
    \item When $m = n$, the graph $K_{m,m} - M_k$, obtained by removing the edges in $M_k$, exhibits pair PST.
    \item When $m > n$, let $E$ be a set of edges added within the larger partite set containing $a$ and $b$, connecting both $a$ and $b$ to $m - n$ additional vertices in such a way that the involution $(a~~b)(c~~d)$ is preserved. Then the graph $K_{m,n} - M_k + E$ exhibits pair PST.
\end{enumerate}
\end{thm}

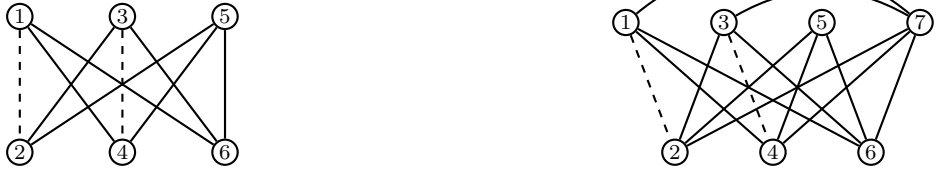
\begin{figure}
\centering

\begin{minipage}[t]{0.47\textwidth}
\centering

\begin{tikzpicture}[scale=0.45]
\node[circle,thick,draw,inner sep=1pt] (A) at (-3,0) {\scriptsize$1$};
\node[circle,thick,draw,inner sep=1pt] (B) at (0,0) {\scriptsize$3$};
\node[circle,thick,draw,inner sep=1pt] (C) at (3,0) {\scriptsize$5$};
\node[circle,thick,draw,inner sep=1pt] (D) at (-3,-4) {\scriptsize$2$};
\node[circle,thick,draw,inner sep=1pt] (E) at (0,-4) {\scriptsize$4$};
\node[circle,thick,draw,inner sep=1pt] (F) at (3,-4) {\scriptsize$6$};

 \draw[thick,dashed] (A) -- (D);
 \draw[thick] (A) -- (E);
 \draw[thick] (A) -- (F);
 \draw[thick] (B) -- (D);
 \draw[thick,dashed] (B) -- (E);
 \draw[thick] (B) -- (F);
 \draw[thick] (C) -- (D);
 \draw[thick] (C) -- (E);
 \draw[thick] (C) -- (F);

\end{tikzpicture}
\end{minipage}
\hfill
\begin{minipage}[t]{0.46\textwidth}
\centering

\begin{tikzpicture}[scale=0.43]
\node[circle,thick,draw,inner sep=1pt] (A) at (-3,0) {\scriptsize$1$};
\node[circle,thick,draw,inner sep=1pt] (B) at (0,0) {\scriptsize$3$};
\node[circle,thick,draw,inner sep=1pt] (C) at (3,0) {\scriptsize$5$};
\node[circle,thick,draw,inner sep=1pt] (D) at (6,0) {\scriptsize$7$};
\node[circle,thick,draw,inner sep=1pt] (E) at (-1.5,-4) {\scriptsize$2$};
\node[circle,thick,draw,inner sep=1pt] (F) at (1.5,-4) {\scriptsize$4$};
\node[circle,thick,draw,inner sep=1pt] (G) at (4.5,-4) {\scriptsize$6$};


\draw[thick,dashed] (A) -- (E);
\draw[thick] (A) -- (F);
\draw[thick] (A) -- (G);
\draw[thick] (B) -- (E);
\draw[thick,dashed] (B) -- (F);
\draw[thick] (B) -- (G);
\draw[thick] (C) -- (E);
\draw[thick] (C) -- (F);
\draw[thick] (C) -- (G);
\draw[thick] (D) -- (E);
\draw[thick] (D) -- (F);
\draw[thick] (D) -- (G);
\draw[thick,bend left=40] (A) to (D);
\draw[thick,bend left=30] (B) to (D);

\end{tikzpicture}
\end{minipage}
\caption{Edge-perturbed complete bipartite graphs: $(i)~ K_{3,3}-M_k,~ (ii)~K_{4,3}-M_k+E$.}
\label{f3}
\end{figure}
In \cite{pal9}, various classes of graphs are identified that exhibit pair PST with respect to the adjacency, Laplacian, and signless Laplacian matrices between the same pair states at the same time. It follows from Theorem \ref{T3} that the graphs depicted in Figure \ref{f3} exhibit pair PST with respect to the $q$-Laplacian matrix between $\frac{1}{\sqrt{2}}\left( \mathbf{e}_1 - \mathbf{e}_3 \right)$ and $\frac{1}{\sqrt{2}}\left( \mathbf{e}_2 - \mathbf{e}_4 \right)$ at time $\frac{\pi}{2q}$. Moreover, by \cite[Theorem 2]{pal7}, the graphs described in Theorem \ref{T3} also exhibit pair PST with respect to the adjacency matrix. Consequently, Theorem \ref{T3} characterizes a family of non-regular graphs that exhibit pair PST with respect to the adjacency, Laplacian, and signless Laplacian matrices between the same pair states at the same time.

\subsection{Cycles}
Let $C_n$ denote the cycle on $n$ vertices with vertex set $\mathbb{Z}_n$, where vertices $i$ and $j$ are adjacent if and only if $ (i-j) \equiv \pm1 \pmod{n}$. In what follows, we explore whether edge-perturbed cycles, with or without potential, admit pair PST with respect to the $q$-Laplacian matrix. By Corollary~\ref{C1}, pair PST relative to the $q$-Laplacian matrix occurs in the cycles $C_4$, $C_6$, and $C_8$, as there is vertex PST relative to their respective $\mathscr{L}_{-}$ matrices. Since cycles are regular graphs, it follows from \cite{kim} that only these three cycles admit pair PST with respect to the $q$-Laplacian matrix. Here, we investigate state transfer in edge-perturbed cycles with or without potentials. Although $C_5$ does not admit pair PST, Corollary~\ref{C1} shows that $C_5$ equipped with a potential $1$ at vertices $1$ and $4$ yields Laplacian pair PST between $\frac{1}{\sqrt{2}}(\e_1 - \e_4)$ and $\frac{1}{\sqrt{2}}(\e_2 - \e_3)$ at time $\frac{\pi}{2}$. A single edge perturbation of a cycle $C_n$ amounts to a rank-one positive semidefinite update of the Laplacian and signless Laplacian matrices. Let $\zeta$ be defined as $\zeta = -1$ when referring to the Laplacian matrix, and $\zeta = 1$ when referring to the signless Laplacian matrix. We state below the normalized eigenvectors of $C_n$ as given in \cite[Lemma 6.4]{god25}.

\begin{lem}\label{L3}
The Laplacian and signless Laplacian eigenvalues of $C_n$ are \[\theta_j= 2+2\zeta\cos\ob{\frac{2j\pi}{n}},\] where $0 \leq j \leq \lfloor \frac{n}{2} \rfloor.$ The eigenvector corresponding to eigenvalue $\theta_0=2\ob{1+\zeta}$ is $v_0= \frac{1}{\sqrt{n}}[1,1,\ldots,1]^T,$ while the eigenvector for $\theta_{\frac{n}{2}}=2\ob{1-\zeta},$ whenever $n$ is even is $v_{\frac{n}{2}}=\frac{1}{\sqrt{n}}[1,-1,1,\ldots,1,-1]^T.$ For $1 \leq j < \frac{n}{2},$ the following are eigenvectors for $\theta_j:$
$$v_j=\sqrt{\frac{2}{n}}\tb{1 ~~\cos\ob{\frac{2j\pi}{n}} ~~ \cos\ob{\frac{4j\pi}{n}} ~~ \ldots~~ \cos\ob{\frac{2j(n-1)\pi}{n}}}^T $$
and $$v_{n-j}=\sqrt{\frac{2}{n}}\tb{0 ~~\sin\ob{\frac{2j\pi}{n}} ~~ \sin\ob{\frac{4j\pi}{n}} ~~ \ldots~~ \sin\ob{\frac{2j(n-1)\pi}{n}}}^T.$$ Moreover, $\cb{v_0,\ldots, v_{n-1}}$ is an orthonormal basis for $\mathbb{R}^n.$
\end{lem}

Let $b \in \mathbb{Z}_n\setminus \{0\}$ and $\rho \in \mathbb{R}^{+}$. If an edge of weight $\rho$ is added in $C_n$ between the vertices $0$ and $b$, then the Laplacian and signless Laplacian matrices of the resulting graph $C_n+\rho~\{0,b\}$ is given by $$M_{\rho}=M+\rho~\w\w^T, \quad \w=\ob{\e_0+\zeta\e_b} \text{ and } M\in\cb{L, Q}.$$
Since $\theta_j,$ where $1 \leq j < \frac{n}{2}$ has multiplicity $2$ as an eigenvalue of $C_n$, it follows from Theorem \ref{T1} that $\theta_j$ is also an eigenvalue for the edge perturbed graph $C_n+\rho~\{0,b\}.$ Next we find an eigenvector of $C_n+\rho~\{0,b\}$ corresponding to $\theta_j.$ If $c_1$ and $c_2$ are scalars, not both zero, then the vector $c_1 v_j + c_2 v_{n-j}$ is also an eigenvector of $M$ corresponding to the eigenvalue $\theta_j$. Then  
$$M_{\rho}\ob{c_1v_j+c_2v_{n-j}} =M\ob{c_1v_j+c_2v_{n-j}}+\rho\ob{c_1\w^Tv_j+c_2\w^Tv_{n-j}}\w,$$

where $\w^Tv_j=\sqrt{\frac{2}{n}}\ob{1+\zeta\cos\ob{\frac{2bj\pi}{n}}} \quad \text{and} \quad \w^Tv_{n-j}=\sqrt{\frac{2}{n}}~\zeta \sin \ob{\frac{2bj\pi}{n}}.$ An eigenvector of $M_{\rho}$ corresponding to $\theta_j$ is obtained as
\[\z_j=\begin{cases}
v_j+v_{n-j},& \text{whenever }\w^Tv_{n-j}= 0 =\w^Tv_j,\\
\ob{\w^Tv_{n-j}}v_j- \ob{\w^Tv_j}v_{n-j},& \text{otherwise.}
\end{cases}
\]
It follows that if $\w^Tv_{n-j}= 0 =\w^Tv_j$ and $k \in \mathbb{Z}_n,$ then
\begin{equation}\label{E5}
  \e_k^T \z_j= \dfrac{2}{\sqrt{n}} \cos{\left( \dfrac{2jk\pi}{n}-\dfrac{\pi}{4} \right)}. 
\end{equation}
In the remaining cases
\begin{equation}\label{E6}
\e_k^T\z_j=\begin{cases}
-\dfrac{4}{n} \cos{\left( \dfrac{bj\pi}{n} \right)} \sin{\left( \dfrac{(2k - b)j\pi}{n} \right)}, & \text{whenever } \zeta = 1, \\
\\
-\dfrac{4}{n} \sin{\left( \dfrac{bj\pi}{n} \right)} \cos{\left( \dfrac{(2k - b)j\pi}{n} \right)}, & \text{whenever } \zeta = -1.
\end{cases}
\end{equation}
 The following result provides a characterization of vertex PST in the edge-perturbed cycle for both Laplacian and signless Laplacian matrices.
\begin{lem}\label{L4}
Let $\rho \in \mathbb{Z}^+$, $b \in \mathbb{Z}_n\setminus \cb{0}$ and let $C_n$ be a cycle on $n \geq 7$ vertices. If the graph $C_n + \rho ~\{0, b\}$ exhibits vertex PST, then the following hold: 
\begin{enumerate}
  
  \item In the Laplacian case, both $\frac{b}{2} + \frac{n}{4}$ and $\frac{b}{2} + \frac{3n}{4}$ must be integers, and PST occurs only between these two vertices.

   \item In the signless Laplacian case, both $\frac{b}{2}$ and $\frac{n + b}{2}$ must be integers whenever $b \neq \frac{n}{2}$; and both $\frac{3n}{8}$ and $\frac{7n}{8}$ must be integers whenever $b = \frac{n}{2}.$ In either case, PST occurs only between these two vertices.
   
\end{enumerate}
\end{lem}

\begin{proof}
Since $0 \leq \frac{(2k - b)\pi}{n} < 2\pi$, it follows from \eqref{E5} and \eqref{E6} that $\e_k^T \z_1 = 0$ under the following conditions: when $\zeta = 1$ and $b \neq \frac{n}{2},$ this holds if $k = \frac{b}{2}$ or $k = \frac{n + b}{2}$; and when $\zeta = 1$ and $b = \frac{n}{2},$ this holds if $k=\frac{3n}{8}$ or $k=\frac{7n}{8}.$ In the case when $\zeta = -1$, the term $\e_k^T \z_1=0$ if $k = \frac{b}{2} + \frac{n}{4}$ or $k = \frac{b}{2} + \frac{3n}{4}$. If $\e_k^T \z_1 \neq 0,$ then the eigenvalue $\theta_1= 2+2\zeta \cos \ob{\frac{2\pi}{n}}$ is in the support of the vertex $k.$ Note that $\cos\ob{\frac{2\pi}{n}}$ is an algebraic integer of degree $\frac{\phi(n)}{2},$ where $\phi(n)$ denote the Euler totient function. Therefore, the eigenvalue $\theta_1$ is not an integer whenever $n \geq 7.$ By \cite[Theorem 4]{kirk4}, if the graph $C_n + \rho~\{0, b\}$ exhibits vertex PST, then the following conditions must hold: in the signless Laplacian case (i.e., $\zeta = 1$), when $b \neq \frac{n}{2},$ both $\frac{b}{2}$ and $\frac{n + b}{2}$ must be integers and PST occurs only between them; and when $b = \frac{n}{2},$ both $\frac{3n}{8}$ and $\frac{7n}{8}$ must be integers and PST occurs only between them. In the Laplacian case (i.e., $\zeta = -1$), both $\frac{b}{2} + \frac{n}{4}$ and $\frac{b}{2} + \frac{3n}{4}$ must be integers and PST occurs between these two vertices.
\end{proof}
The next result establishes the non-existence of vertex PST in edge-perturbed cycles, except for a few cases, with respect to both Laplacian and signless Laplacian matrices.
\begin{thm}\label{T4}
Let $\rho \in \mathbb{Z}^+$, $b \in \mathbb{Z}_n\setminus \cb{0}$ and let $C_n$ be a cycle. The graph $C_n + \rho~(0,b)$
does not admit perfect vertex state transfer with respect to the Laplacian
for $n \geq 15$, and with respect to the signless Laplacian for $n \geq 9$.
\end{thm}
\begin{proof}
Consider the Laplacian case. If $\frac{b}{2} + \frac{n}{4}$ and $\frac{b}{2} + \frac{3n}{4}$ are integers, then from \eqref{E5} and \eqref{E6}, the eigenvalues $\theta_2$ and $\theta_3$ lie in the eigenvalue support of both vertices $\frac{b}{2} + \frac{n}{4}$ and $\frac{b}{2} + \frac{3n}{4}$. By \cite[Theorem 4]{kirk4}, both $\frac{b}{2} + \frac{n}{4}$ and $\frac{b}{2} + \frac{3n}{4}$ are not periodic as
\begin{equation*}
|\theta_3-\theta_2| = 4\left|\sin\ob{\frac{5\pi}{n}} \sin\ob{\frac{\pi}{n}} \right|
    <1,\quad \text{whenever }n \geq 15.
\end{equation*}
In the case of signless Laplacian, if $C_n + \rho~\{0,b\}$ has an odd number of vertices, then the claim follows directly from Lemma \ref{L4}. When $n\equiv 0\pmod 8$ and $b=\frac{n}{2},$ we deduce from \eqref{E5} and \eqref{E6} that both eigenvalues $\theta_2$ and $\theta_3$ lie in the eigenvalue support of the vertices $\frac{3n}{8}$ and $\frac{7n}{8}$. Hence, there is no vertex PST in this case as well. When both $n$ and $b~(\neq \frac{n}{2})$ are even, the graph $C_n + \rho~\cb{0,b}$ admits an involution $\phi$ that maps vertex $0$ to $b$ and fixing only the vertices $\frac{b}{2}$ and $\frac{n+b}{2}.$ For every choice of $b\neq \frac{n}{2}$, we obtain a half graph that is a path on $\frac{n-2}{2}$ vertices, with a vertex $j$ of degree $2+\rho$, for some $j<\lfloor \frac{n-2}{4}\rfloor$. Now, we use Corollary \ref{C1} to show the absence of vertex PST between the fixed vertices. When $q=-1,$ the matrix $\widetilde{\mathscr{L}}_+$, as described in Theorem \ref{T2}, is permutation-similar to a tridiagonal matrix with diagonal entries $2,$ except for a single entry $2+2\rho.$
Since the tridiagonal matrix is not symmetric about the anti-diagonal, 
it follows from \cite[Lemma 2]{kay2} that the path representing the tridiagonal matrix does not have vertex PST between its end vertices. Accordingly, vertex PST does not occur in $C_n + \rho~\{0,b\}$ between $\frac{b}{2}$ and $\frac{n+b}{2}.$ 
\end{proof}
Next we investigate the existence of pair PST in a cycle with a single edge perturbation relative to both Laplacian and signless Laplacian matrices.
\begin{lem}\label{L5}
Let $\rho \in \mathbb{Z}^+$, $b \in \mathbb{Z}_n\setminus \cb{0}$ and let $C_n$ be a cycle on $n \geq 13$ vertices. If $C_n+\rho~ \cb{0,b}$ exhibits pair PST from $\frac{1}{\sqrt{2}}(\e_k-\e_l),$ then the following hold:
\begin{enumerate}
    \item In the Laplacian case, $k+l$ must be either $b$ or $n+b.$
    
    \item In the signless Laplacian case:
    \begin{enumerate}
        \item If \( b \ne \frac{n}{2} \), then \( n \equiv 0 \pmod{2} \) and \( k + l = b + \frac{n}{2} \) or \( k + l = b + \frac{3n}{2} \). \item If \( b = \frac{n}{2} \), then \( n \equiv 0 \pmod{4} \) and \( k + l = \frac{n}{4} \) or \( k + l = \frac{5n}{4} \).
    \end{enumerate}
\end{enumerate}
\end{lem}
\begin{proof}
If $C_n$ is a cycle on $n \geq 13$ vertices, then the eigenvalue $\theta_1=2+2\zeta \cos\ob{\frac{2\pi}{n}}$ is not a quadratic integer. Accordingly, by \cite[Theorem 3.2]{god25}, there is no pair PST in $C_n+\rho~ \cb{0,b}$ from the pair state $\frac{1}{\sqrt{2}}(\e_k-\e_l)$ whenever $\theta_1$ belongs to its eigenvalue support. It follows from \eqref{E5} and \eqref{E6} that
\[
(\e_k-\e_l)^T\z_1=\begin{cases}
\dfrac{-4}{\sqrt{n}}~ \sin{\left( \dfrac{(k+l)\pi}{n}-\dfrac{\pi}{4}\right)}\sin{\left( \dfrac{(k-l)\pi}{n}\right)}, & \zeta = 1 \text{~and~} b=\dfrac{n}{2}, \\ \\
-\dfrac{8}{n}~ \cos{\left( \dfrac{b\pi}{n} \right)} \cos{\left( \dfrac{(k+l-b)\pi}{n} \right)}\sin{\left( \dfrac{(k-l)\pi}{n} \right)}, & \zeta = 1 \text{~and~} b\neq \dfrac{n}{2}, \\
\\
\dfrac{8}{n}~ \sin{\left( \dfrac{b\pi}{n} \right)} \sin{\left( \dfrac{(k+l-b)\pi}{n} \right)}\sin{\left( \dfrac{(k-l)\pi}{n} \right)}, &  \zeta = -1.

\end{cases}
\]
Note that \( \sin\left( \frac{(k - l)\pi}{n} \right) \neq 0 \) since \( 0 < k - l < n \). As \( 0 \leq k + l - b < 2n \), we have \( (\e_k - \e_l)^T \z = 0 \) under the following conditions. When \( \zeta = 1 \), the expression vanishes if \( b \neq \frac{n}{2} \) and \( k + l = b + \frac{n}{2} \) or \( b + \frac{3n}{2} \), or if \( b = \frac{n}{2} \) and \( k + l = \frac{n}{4} \) or \( \frac{5n}{4} \). When \( \zeta = -1 \), the expression vanishes if \( k + l = b \) or \( n + b \).
\end{proof}
Now we include some examples of edge-perturbed cycles that exhibit pair PST with respect to Laplacian and signless Laplacian matrices.
\begin{exm}
The edge perturbed cycle $C_n+\rho~\cb{0,b}$ exhibits perfect pair state transfer in the following cases:
	\begin{enumerate}
		\item $n=3$ and $b=1.$ If $\rho$ is a positive integer, then Laplacian PST occurs between $\frac{1}{\sqrt{2}}(\e_0-\e_2)$ and $\frac{1}{\sqrt{2}}(\e_1-\e_2)$ at time $\frac{\pi}{2\rho}.$
		\item $n=4$ and $b=1.$ If $\rho=1$, then signless Laplacian PST occurs between $\frac{1}{\sqrt{2}}(\e_0-\e_1)$ and $\frac{1}{\sqrt{2}}(\e_2-\e_3)$ at time $\frac{\pi}{2}.$ 
		\item $n=4$ and $b=2.$ If $\rho=1$, then Laplacian PST occurs between $\frac{1}{\sqrt{2}}(\e_0-\e_1)$ and $\frac{1}{\sqrt{2}}(\e_0-\e_3)$ at time $\frac{\pi}{2}.$ 
		\item $n=4$ and $b=2.$ If $\rho=2$, then Laplacian PST occurs between $\frac{1}{\sqrt{2}}(\e_0-\e_1)$ and $\frac{1}{\sqrt{2}}(\e_2-\e_3)$ at time $\frac{\pi}{2}.$ 
	\end{enumerate}

\end{exm}
The following observation shows the non-existence of pair PST in the edge-perturbed cycle with respect to the signless Laplacian matrix. 

\begin{thm}
Let $\rho \in \mathbb{Z}^+$ and $b \in \mathbb{Z}_n\setminus \cb{0}.$ The graph $C_n+\rho~\cb {0,b}$ does not exhibit perfect pair state transfer with respect to the signless Laplacian matrix for $n \geq 16$ whenever $b=\frac{n}{2},$ and for $n \geq 22$ whenever $b\neq \frac{n}{2}$.
\end{thm}

\begin{proof}
In the case $b=\frac{n}{2},$ the graph $C_n+\rho~\cb {0,b}$ admits an involution $\phi$ that maps the vertex $0$ to $b$. According to Lemma \ref{L5}, if the graph exhibits pair PST then $n \equiv 0 \pmod{4}$ and $k + l = \frac{n}{4}$  or $ k + l = \frac{5n}{4}$. The half graph induced by the involution $\phi$ is a path on $\frac{n}{2}$ vertices. When $q=-1,$ the matrix $\mathscr{L}_-$ becomes $2I-A(P_{\frac{n}{2}}),$ where $A(P_{\frac{n}{2}})$ is the adjacency matrix of the path on $\frac{n}{2}$ vertices. By \cite[Theorem 9]{cou2024}, there is no vertex PST with respect to $A(P_{\frac{n}{2}})$ whenever $n\geq 8$. Hence, by Corollary \ref{C1}, no pair PST occurs in this case whenever $n \geq 16.$

In the case $b \neq \frac{n}{2},$ it follows from \eqref{E5} and \eqref{E6} that the eigenvalues $\theta_2$ and $\theta_4$ belong to the eigenvalue support of $\frac{1}{\sqrt{2}}(\e_k-\e_l)$, where $k+l =b+\frac{n}{2}$ or $b+\frac{3n}{2}.$ Then
	$$|\theta_4-\theta_2| = 4\left|\sin\ob{\frac{6\pi}{n}} \sin\ob{\frac{2\pi}{n}} \right|
	\leq \frac{48\pi^2}{n^2}.$$
	Note that $|\theta_4-\theta_2| <1,$ whenever $n \geq 22.$ Using \cite[Corollary 3.4]{kim}, we conclude that no pair state transfer occurs in this case.
\end{proof}
In the case of Laplacian we have the following observation.
\begin{rem}
Let $\rho \in \mathbb{Z}^+$, $b \in \mathbb{Z}_n\setminus \cb{0}$ and let $C_n$ be a cycle on $n \geq 13$ vertices. The graph $C_n+\rho~\cb {0,b}$ admits an involution $\phi$ that maps vertex $0$ to $b$, where $a+\phi(a)\equiv b\pmod{n}$ for all $a\in\mathbb{Z}_n$. By Lemma \ref{L5}(1), it is enough to investigate pair PST from $\frac{1}{\sqrt{2}}(\e_a-\e_{\phi(a)}),$ where $a$ is a vertex of the half graph induced by the involution $\phi$. In the case $n=2m,$ the involution $\phi$ fixes exactly two vertices of $C_n+\rho~\cb {0,b}$, namely $\frac{b}{2}$ and $\frac{n+b}{2}$, otherwise fixes only two edges. Then the corresponding half graph is a path on $m-1$ vertices whenever $b$ is even, otherwise a path on $m$ vertices. The matrix $\mathscr{L}_-$, as stated in Corollary~\ref{C1}, is the Laplacian matrix of a path in which a potential $2\rho$ is placed at a vertex $j < m$, where the potentials at the end vertices are $1$ when $b$ is even, $2$ when $b\neq 1$ odd, and all other vertices have no potential. For $b = 1$, one end vertex has potential $2\rho+2$ and the other has potential $2$, with no potential at the remaining vertices. Thus, vertex PST on the path with potentials (representing the half-graph) is equivalent to pair PST in $C_n+\rho~\cb {0,b}$ whenever $n$ is even. If $n$ is odd, the involution $\phi$ fixes one vertex and one edge, and an analogous conclusion holds in this case as well.
\end{rem}
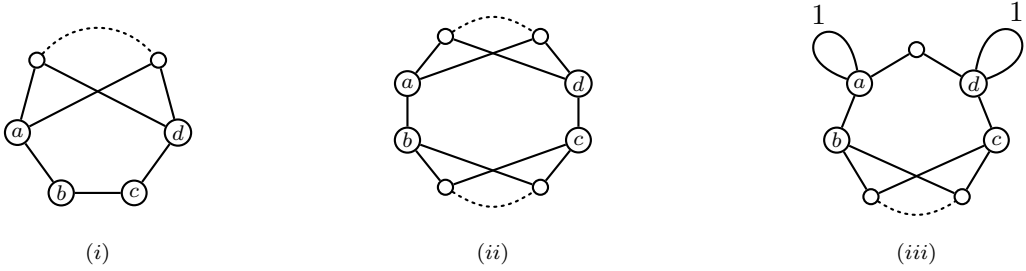
\begin{figure}
\centering
	\begin{minipage}[t]{0.3\textwidth}
		\centering
		\begin{tikzpicture}[scale=0.32]
			\node[circle,thick,draw,inner sep=2pt] (A) at (1,1) {};
			\node[circle,thick,draw,inner sep=1.3pt] (B) at (0.2,-2) {\scriptsize $a$};
			\node[circle,thick,draw,inner sep=1pt] (C) at (2,-4.5) {\scriptsize $b$};
			\node[circle,thick,draw,inner sep=1.5pt] (D) at (5,-4.5) {\scriptsize $c$};
			\node[circle,thick,draw,inner sep=1pt] (E) at (6.8,-2) {\scriptsize $d$};
			\node[circle,thick,draw,inner sep=2pt] (F) at (6,1) {};

			\draw[thick] (A) -- (B);
			\draw[thick] (B) -- (C);
			\draw[thick] (C) -- (D);
			\draw[thick] (D) -- (E);
			\draw[thick] (E) -- (F);
			
			\draw[thick] (A) -- (E);
			\draw[thick] (F) -- (B);
			\draw[dotted, line cap=round, line width=0.8pt] (A) .. controls (2.5,2.7) and (4.5,2.7) .. (F);
			
			\node at (3.5,-7) {\scriptsize$(i)$};
		\end{tikzpicture}
	\end{minipage}
\hfill
	\begin{minipage}[t]{0.3\textwidth}
		\centering
		\begin{tikzpicture}[scale=0.25]
			\node[circle,thick,draw,inner sep=1.3pt] (A) at (0,1) {\scriptsize $a$};
			\node[circle,thick,draw,inner sep=1pt] (B) at (0,-2) {\scriptsize $b$};
			\node[circle,thick,draw,inner sep=2pt] (C) at (2,-4.5) {};
			\node[circle,thick,draw,inner sep=2pt] (D) at (7,-4.5) {};
			\node[circle,thick,draw,inner sep=1.5pt] (E) at (9,-2) {\scriptsize $c$};
			\node[circle,thick,draw,inner sep=1pt] (F) at (9,1) {\scriptsize $d$};
			\node[circle,thick,draw,inner sep=2pt] (G) at (2,3.5) {};
			\node[circle,thick,draw,inner sep=2pt] (H) at (7,3.5) {};
			
			\draw[thick] (A) -- (B);
			\draw[thick] (B) -- (C);
			\draw[thick] (A) -- (H);
			\draw[thick] (D) -- (E);
			\draw[thick] (E) -- (F);
			\draw[thick] (F) -- (H);
			\draw[thick] (A) -- (G);
			\draw[thick] (F) -- (G);
			\draw[thick] (B) -- (D);
			\draw[thick] (C) -- (E);
			
			\draw[dotted, line cap=round, line width=0.8pt] 
			(G) .. controls (4,4.8) and (5,4.8) .. (H);
			\draw[dotted, line cap=round, line width=0.8pt] 
			(C) .. controls (4,-5.8) and (5,-5.8) .. (D);
			
			\node at (4.5,-8) {\scriptsize$(ii)$};
		\end{tikzpicture}
	\end{minipage}
    \hfill
	\begin{minipage}[t]{0.35\textwidth}
		\centering
		\begin{tikzpicture}[scale=0.3]
			\node[circle,thick,draw,inner sep=1.3pt] (A) at (1,0.5) {\scriptsize $a$};
			\node[circle,thick,draw,inner sep=1pt] (B) at (0,-2) {\scriptsize $b$};
			\node[circle,thick,draw,inner sep=2pt] (C) at (1.5,-4.5) {};
			\node[circle,thick,draw,inner sep=2pt] (D) at (5.5,-4.5) {};
			\node[circle,thick,draw,inner sep=1.5pt] (E) at (7,-2) {\scriptsize $c$};
			\node[circle,thick,draw,inner sep=1pt] (F) at (6,0.5) {\scriptsize $d$};
			\node[circle,thick,draw,inner sep=2pt] (G) at (3.5,2) {};

			\draw[thick] (A) -- (B);
			\draw[thick] (B) -- (D);
			\draw[thick] (C) -- (E);
			\draw[thick] (D) -- (E);
			\draw[thick] (E) -- (F);
			\draw[thick] (F) -- (G);
			\draw[thick] (A) -- (G);
			\draw[thick] (C) -- (B);
			\draw[dotted, line cap=round, line width=0.8pt] (C) .. controls (3,-5.5) and (4,-5.5) .. (D);
			
			\draw[thick] (A) to[out=160, in=100, looseness=15] node[above] {\small 1} (A);
			\draw[thick] (F) to[out=80, in=20, looseness=15] node[above] {\small 1} (F);
			
			\node at (3.5,-7) {\scriptsize$(iii)$};
			\end{tikzpicture}
	\end{minipage}
	\caption{Pair state transfer on cycles with potentials and additional edges.}
\label{f4}
\end{figure}
The only cycles that admit pair PST with respect to the $q$-Laplacian matrix are $C_4,~C_6,$ and $C_8$. Beyond these cases, pair PST can also be induced in a cycle $C_n $ with $n \geq 6$ vertices by suitably inserting additional edges and potentials. In particular, it implies from Corollary \ref{C1} that the graph in \ref{f4}$(i)$ exhibits Laplacian pair PST between $\frac{1}{\sqrt{2}}(\e_a-\e_d)$ and $\frac{1}{\sqrt{2}}(\e_b-\e_c)$ at time $\frac{\pi}{2}.$ More generally, in \ref{f4}$(i)$, if potentials $\alpha$ and $\beta$ are assigned to the vertices $a,d$ and $b,c,$ respectively, with $\alpha-\beta=\frac{1}{q}-1$, then pair PST between these pair states can be achieved with respect to the $q$-Laplacian matrix. Similarly, the graphs in Figure \ref{f4}$(ii)$ and Figure \ref{f4}$(iii)$ exhibit pair PST between $\frac{1}{\sqrt{2}}(\e_a-\e_d)$ and $\frac{1}{\sqrt{2}}(\e_b-\e_c)$ at time $\frac{\pi}{2}$ relative to the $q$-Laplacian matrix. The following two observations are immediate.
\begin{thm}
Laplacian perfect pair state transfer can be achieved in a cycle with at least six vertices by inserting exactly two suitable edges.
\end{thm}
\begin{thm}
$q$-Laplacian perfect pair state transfer can be achieved in a cycle with at least eight vertices by inserting exactly four suitable edges.
\end{thm}
\subsection{Paths}
Let $P_n$ be a path on $n$ vertices, and the vertices of $P_n$ are labeled consecutively so that the vertices $j$ and $k$ are adjacent whenever $|j-k|=1.$ Suppose $P_n(\omega)$ denotes the path on $n$ vertices having potential $\omega$ only at the end vertices. The $q$-Laplacian matrix $\mathscr{L}$ corresponding to $P_n(\omega)$ is obtained by
$$\mathscr{L}=(1+q^2)I+(\omega -1)q^2(\e_1\e_1^T+\e_n\e_n^T)-qA.$$
The matrix $A+(1-\omega)q(\e_1\e_1^T+\e_n\e_n^T)$ represents the adjacency matrix of $P_n((1-\omega)q)$. The continuous-time quantum walk on $P_n(\omega)$ relative to $\mathscr{L}$ is equivalent to that on $P_n((1-\omega)q)$ relative to the adjacency matrix. Thus, we have the following observation.

\begin{thm}\label{T8}
Let $P_n(\omega)$ be a path on $n$ vertices with potential $\omega$ assigned only at the end vertices. Then $P_n(\omega)$ exhibits $q$-Laplacian PST (respectively, PGST) if and only if $P_n((1-\omega)q)$ exhibits PST (respectively, PGST) with respect to its adjacency matrix.
\end{thm}
If $\omega=0,$ then $P_n(\omega)$ reduces to a simple path on $n$ vertices. It is observed in \cite[Theorem 3.4 and Theorem 3.8]{kempton} that the path $P_n$ equipped with symmetric potential on $n \geq 4$ vertices does not exhibit PST relative to the adjacency matrix between its end vertices. For $n=2,$ the graph $P_2(q)$ admits PST relative to the adjacency matrix between its end vertices. In the case $n=3,$ by employing similar arguments as used in the proof of \cite[Theorem 3.3]{kempton2}, we conclude that $P_3(q)$ exhibits PST relative to the adjacency matrix between its end vertices. The following observation is immediate.

\begin{thm}\label{T9}
The path $P_n$ admits $q$-Laplacian perfect vertex state transfer between its end vertices if and only if $n =2,3$. For $n = 2$, perfect state transfer occurs for any $q$ at time $\frac{\pi}{2q}$. For $n = 3$, perfect state transfer occurs precisely when $q = \sqrt{\frac{8l^2}{k^2 - l^2}}$, where $k$ and $l$ are integers satisfying $k > l$ and $k \not\equiv l \pmod{2}$, and the state transfer time is $ \frac{\pi(k^2 - l^2)}{4l}.$
\end{thm}
In \cite[Theorem 3]{kempton}, it is observed that the path $P_n(\omega)$ exhibits vertex PGST relative to the adjacency matrix between the end vertices with a suitable choice of potential $\omega$. Then, according to Theorem \ref{T8}, we have the following observation.
\begin{thm}
Let $P_n$ be a path on $n$ vertices. Then there exists $q \in \mathbb{R}$ such that $P_n$ admits $q$-Laplacian vertex PGST between its end vertices.
\end{thm}

The only cycles that exhibit $q$-Laplacian pair PST are $C_4,~C_6,$ and $C_8$. For even $n$, there is a non-trivial involution $\phi$ of $C_n$ which fixes either two vertices or two edges, while for odd $n$, it fixes exactly one vertex and one edge.
We use Corollary~\ref{C1} to conclude the following.
\begin{thm}\label{T10}
 Let $P_n(\omega_1,\omega_2)$ denote a path on $n$ vertices with potentials $\omega_1$ and $\omega_2$ assigned only at the end vertices. Then $P_n(\omega_1,\omega_2)$ does not exhibit vertex PST with respect to its $q$-Laplacian matrix in any of the following cases:
	\begin{enumerate}
		\item $n \ge 2$ with $\omega_1=1+\frac{1}{q}$ and $\omega_2=1;$
		\item $n \geq 4$ with $\omega_1=\omega_2=1;$
		\item $n \geq 5$ with $\omega_1=\omega_2=1+\frac{1}{q}.$
    \end{enumerate}
\end{thm}
The case in Theorem~\ref{T10}(2) with $q=1$ corresponds to the Laplacian matrix. The Laplacian matrix of $P_n(1,1)$ can be expressed as $L = 2I - A$, where $A$ denotes the adjacency matrix of $P_n$. Consequently, it follows that $P_n$ does not admit vertex PST relative to the adjacency matrix whenever $n \geq 4,$ which supports the result provided in \cite{christandl}. Similarly, for $q=-1$, the case in Theorem~\ref{T10}(3) corresponds to the signless Laplacian matrix. In this case, the result reduces to the path $P_n$ with $n \geq 5$ vertices, which does not admit vertex PST. This observation is consistent with \cite[Corollary~5]{alvir}.

Next, we include a few observations regarding pair state transfer on the path $P_n(\omega)$. It can be observed that $P_n(\omega)$ admits a non-trivial involution $\phi$ which fixes an edge when $n$ is even, and fixes a vertex when $n$ is odd. The matrix $\mathscr{L}_-$ associated with the involution $\phi$ can be evaluated as 
$$\mathscr{L}_{-} =\begin{cases}
    q^2\omega\e_1\e_1^T+\mathscr{L}(P_k)+(q^2+q)\e_k\e_k^T, & \text{whenever } n=2k,\\
    q^2\omega\e_1\e_1^T+\mathscr{L}(P_k)+q^2\e_k\e_k^T, &\text{whenever } n=2k+1.
\end{cases}$$
Since the paths $P_2$ and $P_3$ admit vertex PST relative to the adjacency matrix, we conclude the following using Corollary \ref{C1}:
\begin{exm}\label{Ex2}
Let $P_n(\omega)$ denote the path on $n$ vertices having potential $\omega$ only at the end vertices. Then the following hold relative to the $q$-Laplacian matrix:
\begin{enumerate}
\item $P_3(1)$ admits PST between $\frac{1}{\sqrt{2}}(\e_1-\e_2)$ and $\frac{1}{\sqrt{2}}(\e_2-\e_3)$ at time $\frac{\pi}{q\sqrt{2}}.$ 
\item $P_4\ob{1+\frac{1}{q}}$ admits PST between $\frac{1}{\sqrt{2}}(\e_1-\e_4)$ and $\frac{1}{\sqrt{2}}(\e_2-\e_3)$ at time $\frac{\pi}{2q}.$
\item $P_5(1)$ admits PST between $\frac{1}{\sqrt{2}}(\e_1-\e_5)$ and $\frac{1}{\sqrt{2}}(\e_2-\e_4)$ at time $\frac{\pi}{2q}.$
\item $P_7(1)$ admits PST between $\frac{1}{\sqrt{2}}(\e_1-\e_7)$ and $\frac{1}{\sqrt{2}}(\e_3-\e_5)$ at time $\frac{\pi}{q\sqrt{2}}.$
\end{enumerate}
\end{exm}
It can be observed that if a tree is attached to $P_5(1)$ in Example \ref{Ex2}(3) as illustrated in Figure \ref{f5}$(i)$, then PST occurs in the resulting graph between the same pair states. Thus, we have the following observation.
\begin{thm}
There exist infinitely many trees with potential at exactly two vertices that admit perfect pair state transfer relative to the $q$-Laplacian matrix.    
\end{thm}
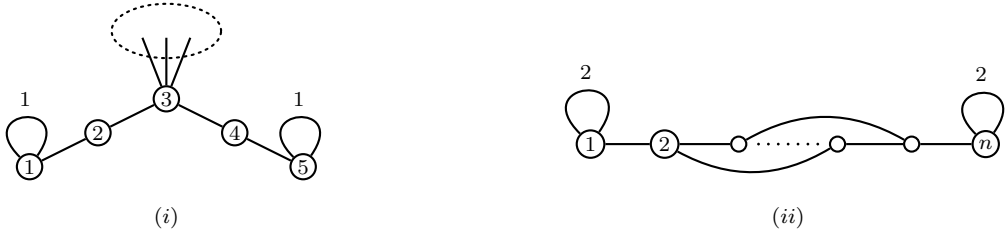
\begin{figure}
	\centering
	\begin{minipage}[t]{0.4\textwidth}
		\centering
		\begin{tikzpicture}[scale=0.45]
			\node[circle,thick, draw,inner sep=1pt] (A) at (-4,0) {\scriptsize $1$};
			\node[circle,thick,draw,inner sep=1pt] (B) at (-2,1) {\scriptsize $2$};
			\node[circle,thick,draw,inner sep=1pt] (C) at (0,2) {\scriptsize $3$};
			\node[circle,thick,draw,inner sep=1pt] (D) at (2,1) {\scriptsize $4$};
			\node[circle,thick,draw,inner sep=1pt] (E) at (4,0) {\scriptsize $5$};
			
			\draw[thick] (A) -- (B);
			\draw[thick] (B) -- (C);
			\draw[thick] (C) -- (D);
			\draw[thick] (D) -- (E);
			
			\draw[thick] (A) to[out=135, in=55, looseness=10] node[above] {\scriptsize $1$} (A);
			\draw[thick] (E) to[out=135, in=55, looseness=10] node[above] {\scriptsize $1$} (E);
			
			\draw[dotted, line cap=round, line width=0.8pt] (0,4) ellipse (1.6cm and 0.8cm);
			\draw[thick] (C) -- (0,3.8);
			\draw[thick] (C) -- (-0.7,3.8);
			\draw[thick] (C) -- (0.7,3.8);
			\node at (0,-1.5) {\scriptsize$(i)$};
			
		\end{tikzpicture}
	\end{minipage}
\hfill
	\begin{minipage}[t]{0.58\textwidth}
		\centering
		\begin{tikzpicture}[scale=0.65]
			\node[circle,thick,draw,inner sep=1.3pt] (A) at (-5,0) {\scriptsize $1$};
			\node[circle,thick,draw,inner sep=1.3pt] (B) at (-3.5,0) {\scriptsize $2$};
			\node[circle,thick,draw,inner sep=2pt] (C) at (-2,0) {};
			\node[circle,thick,draw,inner sep=2pt] (F) at (0,0) {};
			\node[circle,thick,draw,inner sep=2pt] (D) at (1.5,0) {};
			\node[circle,thick,draw,inner sep=1.3pt] (E) at (3,0) {\scriptsize $n$};

			\draw[thick] (A) -- (B);
			\draw[thick] (B) -- (C);
			\draw[thick] (D) -- (F);
			\draw[thick] (D) -- (E);
			\draw[thick,bend left] (F) to (B);
			\draw[thick, bend left] (C) to (D);
			\foreach \x in {-1.6,-1.4,-1.2,-1,-0.8,-0.6,-0.4} {\fill (\x,0) circle (0.8pt); }
			
			\draw[thick] (A) to[out=135, in=55, looseness=10] node[above] {\scriptsize $2$} (A);
			\draw[thick] (E) to[out=135, in=55, looseness=10] node[above] {\scriptsize $2$} (E);
            \node at (-1,-1.5) {\scriptsize$(ii)$};
		\end{tikzpicture}
	\end{minipage}
	\caption{Pair PST in edge-perturbed graphs with potential.}
	\label{f5}
\end{figure}
 It is possible to have $q$-Laplacian pair PST in a path $P_n$ with a few additional edges and assigning appropriate potentials as illustrated in Figure \ref{f5}$(ii)$, where pair PST occurs between $\frac{1}{\sqrt{2}}(\mathbf{e}_1 - \mathbf{e}_n)$ and $\frac{1}{\sqrt{2}}(\mathbf{e}_2 - \mathbf{e}_{n-1})$ at time $\frac{\pi}{2q}$, as follows from Corollary~\ref{C1}.

\begin{thm}
Let $P_n$ be the path on $n \geq 6$ vertices. By adding the edges $\{2,n-2\}$ and $\{3,n-1\}$ and assigning a potential of $2$ to the end vertices $1$ and $n$, the resulting graph admits $q$-Laplacian perfect pair state transfer.
\end{thm}
\section{Acknowledgments}
The authors sincerely thank Professor Sivaramakrishnan Sivasubramanian, IIT Bombay, India, for his valuable suggestions.

\bibliographystyle{abbrv}
\bibliography{Ref}
\end{document}